\documentclass[a4paper,12pt]{article}

\usepackage[top=3.0cm,bottom=3.0cm,left=2.25cm,right=2.25cm]{geometry}%{geometry} 表示可以自定义页面设置
\usepackage{amsmath,amsthm,mathrsfs,graphicx,amsfonts}%{mathrsfs}用于产生一种数学用的花体字，数学公式宏包
\usepackage{bm}%处理数学公式中的黑斜体，公式中的粗体字符（用命令\boldsymbol）
\usepackage{cases}%\begin{numcases}{|x|=}x,&for$x\geq0$\\-x,&for$x<0$\end{numcases}
\usepackage[english]{babel}%算法
\usepackage{amsmath,amsthm}
\usepackage{amsfonts}
\usepackage{latexsym}
\usepackage{graphicx}%图形宏包
\usepackage{txfonts}
\usepackage[numbers,sort&compress]{natbib}
\usepackage[natural]{xcolor}
\usepackage{rotating}
\usepackage{mathtools}
\usepackage{enumitem}

\newtheorem{lem}{Lemma}[section]
\newtheorem{thm}[lem]{Theorem}

\newtheorem{prob}[lem]{Problem}

\newtheorem{conj}[lem]{Conjecture}

\begin{document}
\title{A complement of the Erd\H{o}s-Hajnal problem on paths with
equal-degree endpoints}
\date{}
%%%%%%%%%%%%%%%%%%%%%%%%%%%%%%%%%%%%%%%%%%%%%%%%%%%%
\author{Zhen Liu\footnote{Email: 1552580575@qq.com}, ~Qinghou Zeng\footnote{Research supported by National Key R\&D Program of China (Grant No. 2023YFA1010202) and National Natural Science Foundation of China (Grant No. 12371342). Email: zengqh@fzu.edu.cn (Corresponding
author)}\\
{\small Center for Discrete Mathematics, Fuzhou University, Fujian, 350003, China}}
%%%%%%%%%%%%%%%%%%%%%%%%%%%%%%%%%%%%%%%%%%%%%%%%%%%%%

%%%%%%%%%%%%%%%%%%%%%%%%%%%%%%%%%%%%%%%%%%%%%%%%%%%
\maketitle

\maketitle
\begin{abstract}
Answering a question of Erd\H{o}s and Hajnal,  Chen and Ma proved that for all \(n\geq600\)  every graph with \(2n + 1\) vertices and at least \(n^2 + n+1\) edges contains two vertices of equal degree connected by a path of length three. The complete bipartite graph $K_{n,n+1}$ shows that this edge bound is sharp. In this paper, we develop a novel approach to handle graphs with large equal degrees, which enables us to establish the result for all $n\ge2$, thereby fully resolving the problem posed by Erd\H{o}s and Hajnal.
\end{abstract}

\section{Introduction}\label{Intro}
In graph theory, there is a fundamental fact that every graph with at least two vertices has two vertices with equal degree. Naturally, we wonder what additional properties must these vertices satisfy under certain conditions.
%whether, when the edge density of graph \(G\) is relatively high, there must exist a path connecting two vertices with equal degrees.
In 1991, Erd\H{o}s and Hajnal proposed the following related question (see also Problem \# 816 in Thomas Bloom's collection of Erd\H {o}s problems \cite{Erd816}).

\begin{prob} [Erd\H{o}s \cite{Erd1991}]\label{EH}
Is it true that every \((2n + 1)\)-vertex graph with \(n^{2}+n + 1\)
edges contains two vertices of the same degree that are joined by a path of length three?
\end{prob}

%This problem is also listed as Problem \# 816 in Thomas Bloom's collection of Erd\H {o}s of problems \cite{Erd816}.
The bound on the number of edges would be sharp if true, as shown by the complete bipartite graph \(K_{n,n + 1}\). Recently, Chen and Ma \cite{CM2025} provided a stronger version of a solution to Problem \ref{EH} as follows.

\begin{thm}[Chen and Ma \cite{CM2025}]\label{u3}
Let $n\geq 600$. The unique \((2n + 1)\)-vertex graph with at least \(n^{2}+n\) edges,
that does not contain two vertices of the same degree joined by a path of length three, is the
complete bipartite graph \(K_{n,n + 1}\).
\end{thm}

%The main contribution of this article is the following theorem. Compared with the results of Chen and Ma, we only need \(n > 1\), that is to say, we only need the graph \(G\) to have more than 4 vertices, and this completely solves the Erd\H{os}–Hajnal problem stated above.
In this paper, we show that Theorem \ref{u3} holds for all $n\ge2$. We mention that our method is different and useful for graphs with large equal degrees.
\begin{thm}\label{LZ}
 Let $n\ge2$. The unique \((2n + 1)\)-vertex graph with at least \(n^{2}+n\) edges,
that does not contain two vertices of the same degree joined by a path of length three, is the
complete bipartite graph \(K_{n,n + 1}\).
\end{thm}

Chen and Ma \cite{CM2025} also obtained a similar result for graphs with even number of vertices.
\begin{thm}[Chen and Ma \cite{CM2025}]
There exists an integer \(n_0 > 0\) such that the following holds for all \(n\geq n_0\). The unique \(2n\)-vertex graph with at least \(n^2 - 1\) edges, that does not contain two vertices of the same degree joined by a path of length three, is the complete bipartite graph \(K_{n - 1,n + 1}\).
\end{thm}

Our method can also extends to graphs with an even number of vertices as follows.
%Since the proof of this result is very similar to that of Theorem \ref{LZ}, we omit its detail here and refer the interested reader to \cite{LZ2025} for a proof.
%The proof of this theorem closely resembles that of Theorem \ref{LZ}, so we omit the details here to avoid repetition. Interested readers may find the full argument in \cite{LZ2025}.
\begin{thm}\label{2n}
Let $n\ge3$ be an integer. The unique \((2n)\)-vertex graph with at least \(n^{2}-1\) edges,
that does not contain two vertices of the same degree joined by a path of length three, is the
complete bipartite graph \(K_{n-1,n + 1}\).
\end{thm}

\noindent\textbf{Notation}. Let $G$ be a graph and $\overline{G}$ the complement graph of $G$. For any $v\in V(G)$, denote by $N_{G}(v)$ the set of \emph{neighbors} of $v$ in $G$ and $d_G(v)$ the \emph{degree} of $v$ in $G$. Let $\Delta_G$ be the maximum degree of the vertices in  $G$. For any $S\subseteq V(G)$, let $G[S]$ denote the induced subgraph of $G$ on $S$, and $N_S(v)=N(v)\cap S$ for any $v\in V(G)$.  For each $S, T\subseteq V(G)$, let \(e_G(S)\) denote the number of edges with both endpoints in \(S\), and $e_G(S,T)$ denote the number of edges of $G$ with one end in $S$ and the other end in $T$.  If $S=\{v\}$, then we simply write $e_G(v,T)$ for $e_G(\{v\},T)$. For convenience, we also write $\overline{e}(S)=\binom{|S|}{2}-e_G(S)$ and $\overline{e}(S,T)=|S||T|-e_G(S,T)$. We will drop the reference to $G$ when there is no danger of confusion.
%In addition, we usually write $[k]:=\{1,\ldots,k\}$ for any integer $k\ge2$.

%Let \(G\) be a graph. For a vertex \(x \in V(G)\), let \(N_G(x)\) denote the set of vertices adjacent to \(x\) in \(G\), and let \(\overline{N}_G(x) =V(G) \setminus (\{x\} \cup N_G(x))\) denote the set of vertices not adjacent to \(x\) (excluding \(x\) itself). Let \(d_G(x) = |N_G(x)|\) be the degree of \(x\), and let \(\Delta_G\) denote the maximum degree of vertices in \(G\). For disjoint vertex subsets \(A, B \subseteq V(G)\), let \(e_G(A)\) denote the number of edges with both endpoints in \(A\), and let \(e_G(A, B)\) denote the number of edges with one endpoint in \(A\) and the other in \(B\). When the context is clear, we often drop the subscript \(G\) from these notations. For any $S\subseteq V(G)$, $x\in V(G)$, let $G[S]$ denote the induced subgraph of $G$ on $S$, and $N_S(x)=N(x)\cap S$.
%Note that, for the convenience of readers, instead of redefining the notations, we have adopted the definitions used in Reference \cite{CM2025}.

\section{Proof of Theorem \ref{LZ}}
Throughout this section, for any integer $n\ge2$, let \(G\) be a graph on \(2n + 1\) vertices with at least
\(n^2 + n\) edges such that \(G\) contains no two vertices of the same degree joined by a path of
length three. We aim to prove that \(G\) must be the complete bipartite graph \(K_{n,n + 1}\).
%Suppose, for contradiction, that this is not the case.

We follow the parameter \(\beta\) defined by Chen and Ma \cite{CM2025}. Let \(\beta\) be the largest integer such that \(G\) contains two vertices of degree \(\beta\).
%Since \(G\) must have at least two vertices of the same degree, \(\beta\) is well-defined. The following formula is simple and useful. Moreover, Chen and Ma provided the proof in \cite{CM2025}, so we will not prove it here
%\begin{equation}2 \leq \beta \leq \Delta.\end{equation}
Suppose that $u$ and $v$ are two vertices in $G$ with
$$d(u) = d(v) = \beta,$$
and let $\textbf{1}_{uv}=1$ if $uv\in E(G)$ and $\textbf{1}_{uv}=0$ if $uv\notin E(G)$.
%\left\{
%  \begin{array}{ll}
%    1, & \hbox{$uv\in E(G)$;} \\
%    0, & \hbox{$uv\notin E(G)$.}
%  \end{array}
%\right.\]
Set $B := N(u) \cap N(v)$,
$A_u := N(u)\setminus(\{v\} \cup B)$, $ A_v := N(v)\setminus(\{u\} \cup B)$ and $D :=V(G)\setminus(N(u)\cup N(v)\cup\{u,v\})$.
 Let $|B|=x$. Then
\[
 |A_u| =|A_v|=\beta-x-\textbf{1}_{uv} \;\,\text{and}\;\,  |D|=x-2(\beta-n-\textbf{1}_{uv})-1.
 \]
In what follows, we aim to give a desired lower bound of $e(\overline{G})$ to contradict the fact $e(\overline{G})=\binom{2n+1}{2}-e(G)\le n^2$.

Since $G$ contains no path of length three connecting $u$ and $v$, we have
\begin{equation}\label{ABuv}
e(A_u, B) = e(A_v, B) = e(B) = e(A_u, A_v) = 0.
\end{equation}
It follows that
\begin{align*}
\overline{e}(V(G)\setminus D)&=\overline{e}(u,A_v)+ \overline{e}(v,A_u)+\overline{e}(A_u,A_v)+ \overline{e}(B,A_u)+ \overline{e}(B,A_v)+\overline{e}(B)+\overline{e}(u,v)\\
&=2(\beta-x-\textbf{1}_{uv})+(\beta-x-\textbf{1}_{uv})^2+2x(\beta-x-\textbf{1}_{uv})+\binom{x}{2}+(1-\textbf{1}_{uv})\\
&=(\beta+1-\textbf{1}_{uv})^2-\frac12(x^2+5x)-\textbf{1}_{uv}.
\end{align*}
This together with the fact $\overline{e}(\{u,v\},D)=2|D|$ implies that
\begin{align}\label{Complement-G}
e(\overline{G})&\ge\overline{e}(V(G)\setminus D)+\overline{e}(\{u,v\},D)+\overline{e}(B,D)+\overline{e}(D)\notag\\
&=(\beta+1-\textbf{1}_{uv})^2-\frac12(x^2+5x)-\textbf{1}_{uv}+2(2n+x-2\beta+2\cdot\textbf{1}_{uv}-1)+\overline{e}(B,D)+\overline{e}(D)\notag\\
&=(\beta-1-\textbf{1}_{uv})^2+4n-\frac12(x^2+x)-2-\textbf{1}_{uv}+\overline{e}(B,D)+\overline{e}(D).
\end{align}

We first deal with the situation when $\beta$ is relatively large. For convenience, let $c:=\beta-n-\textbf{1}_{uv}$.
\begin{lem}\label{c-Complement-G}
If $c\ge1$, then we have
\[
e(\overline{G})\ge n^{2}+c^2-c-\textbf{1}_{uv}.
\]
\end{lem}
\begin{proof}[\bf Proof]
Note that $|B|-|D|=2(\beta-n-\textbf{1}_{uv})+1=2c+1.$ If $c\ge1$, then \(|B|-|D|\geq3\). It follows that there exists \(\{u_1, u_2\} \subseteq B\) such that \(|N_{D}(u_1)| = |N_{D}(u_2)|\) by the pigeonhole principle, implying that \(d(u_1)=d(u_2)\) in view of \eqref{ABuv}. Let \(\gamma\) be the largest integer such that there are two vertices in  \(B\) that have exactly \(\gamma\) neighbors in \(D\). Without loss of generality, we may assume that \(u_1\) and \(v_1\) are such two vertices. Clearly, \(u_1\) and \(v_1\) are not adjacent as $e(B)=0$ by \eqref{ABuv}. Similarly, let $B_1 := N_D(u_1) \cap N_D(v_1)$,
$A_{u_1} := N_D(u_1)\setminus B_1$ and $A_{v_1} := N_D(v_1)\setminus B_1$. Write $|B_1|=y$. Then $|A_{u_1}|=|A_{v_1}|=\gamma-y$.
%$|A_{u_1}|=|A_{v_1}|=\gamma-y$ and $D_1 := D\setminus(A_{u_1}\cup A_{v_1}\cup B_1)$, $|D_1|=x-2c-1-2\gamma+y$.
By the assumption, there does not exist a path of length three connecting $u_1$ and $v_1$. Thus, we have
\[
e(A_{u_1}, B_1) = e(A_{v_1}, B_1) = e(B_1) = e(A_{u_1}, A_{v_1}) = 0.
\]
This together with the fact $y\le\gamma$ implies that
\begin{align}\label{D-complement}
\overline{e}(D) &\geq \overline{e}(A_{u_1}, B_1) + \overline{e}(A_{v_1}, B_1) + \overline{e}(B_1) + \overline{e}(A_{u_1}, A_{v_1})\notag\\
&= y(\gamma-y)+y(\gamma-y)+\binom{y}{2}+(\gamma-y)(\gamma-y)\notag\\
&= \gamma^{2}-\frac{1}{2}y^{2}-\frac{1}{2}y\ge \binom{\gamma}{2}.
\end{align}
By the definition of \(\gamma\), we deduce that
\begin{align*}
    e(B,D)= \sum_{v \in B} |N_D(v)| &\leq  \gamma\left(|B|-(|D|-\gamma)\right) + \sum_{k = \gamma+1}^{|D|} k  \\
&=\gamma|B|+\frac12(|D|-\gamma)(|D|-\gamma+1).
\end{align*}
Note that $|B|=x$ and $|D|=x-2c-1$. Thus
\begin{align}\label{BD-complement}
\overline{e}(B,D) &=|B||D|-e(B, D)\notag\\
%&\ge |B|(|D|-\gamma)-\frac12(|D|-\gamma)(|D|-\gamma+1)\\
&\ge\frac12(|D|-\gamma)(2|B|-|D|+\gamma-1)\notag\\
&=\binom{x+1}{2}-\binom{\gamma}{2}-2c^2-c-(2c+1)\gamma-x\notag\\
&\ge\binom{x+1}{2}-\binom{\gamma}{2}+2c^2+3c+1-2(c+1)x,
%&=\binom{x}{2}-(\beta-n-\textbf{1}_{uv})^2-(2\gamma+1)(\beta-n-\textbf{1}_{uv})-\binom{\gamma+1}{2}\\
%&=\binom{x+1}{2}-x-(\beta-n-\textbf{1}_{uv})^2-(2\gamma+1)(\beta-n-\textbf{1}_{uv})-\binom{\gamma}{2}-\gamma\\
%&=\binom{x+1}{2}-\binom{\gamma}{2}-(\beta-n-\textbf{1}_{uv})^2-(\beta-n-\textbf{1}_{uv})-\gamma(2(\beta-n-\textbf{1}_{uv})+1)-x
%\\&\geq (x-2c-1)x-(2c + 1 + \gamma)\gamma-\frac{(x - 2c + \gamma)(x - 2c - 1 - \gamma)}{2}\\
%&=\frac{1}{2}x^{2}-\frac{1}{2}x - 2c^{2}-2\gamma c - c-\frac{1}{2}\gamma^{2}-\frac{1}{2}\gamma.
\end{align}
where the last inequality holds due to $\gamma\le|D|=x-2c-1$.
Combining \eqref{Complement-G}, \eqref{D-complement} and \eqref{BD-complement}, we have
\begin{align}\label{e-Complement-G}
e(\overline{G})&\ge\overline{e}(V(G)\setminus D)+\overline{e}(\{u,v\},D)+\overline{e}(B,D)+\overline{e}(D)\notag\\
&\ge(\beta-1-\textbf{1}_{uv})^2+4n-2-\textbf{1}_{uv}+2c^2+3c+1-2(c+1)x\notag\\
&\ge n^{2}+c^2-c-\textbf{1}_{uv},
\end{align}
where the last inequality follows from the fact $x\le\beta-\textbf{1}_{uv}=n+c$. Thus, we complete the proof of this lemma.
\end{proof}

\begin{lem}\label{Large-equal-degree}
For any integer $n\ge2$, we have $\beta \leq n + 1$. If $\beta=n + 1$, then $G$ is isomorphic to \(K_{n,n + 1}\).
%Let $n>1$. We have $\beta \leq n + 1$. In addition, when \(\beta=n + 1\), the graph \(G\) is \(K_{n,n + 1}\).
\end{lem}
\begin{proof}[\bf Proof]
Suppose that $\beta\ge n+1$. We aim to prove that  $G$ is isomorphic to \(K_{n,n + 1}\). Let $u$ and $v$ are two vertices in $G$ such that $d(u) = d(v) = \beta$. Since $\beta\ge n+1$,  we have $c=\beta-n-\textbf{1}_{uv}\ge1-\textbf{1}_{uv}$. If $\textbf{1}_{uv}=0$, then $e(\overline{G})>n^2$ by Lemma \ref{c-Complement-G} unless $c=1$, a contradiction. If $c=1$, then we have $e(\overline{G})=n^2$. Thus, we conclude that \(x=n + c=n+1\) and \(\gamma=x-2c-1 =n-2\) by combining \eqref{BD-complement} and \eqref{e-Complement-G}, which further implies that $e(B,D)=|B||D|$. Thus, \(G\) is isomorphic to \(K_{n,n + 1}\). If $\textbf{1}_{uv}=1$, then $|B|-|D|\ge3$ unless $c=0$. Thus, there exist two vertices \(u_1, v_1\in B\) such that \(|N_{D}(u_1)| = |N_{D}(v_1)|\) by the pigeonhole principle, implying that \(d(u_1)=d(v_1)\) in view of \eqref{ABuv}. Clearly, \(u_1uvv_1\) is a path of length three with equal-degree endpoints. This leads to a contradiction.

Now, it suffices to assume that $\textbf{1}_{uv}=1$ and $c=0$; moreover, the degrees of the vertices in \(B\) are pairwise distinct. Recall that $|B|-|D|=2c+1$. It follows that $x=|B|=|D|+1$. This means that there is an order $w_1,\ldots,w_x$ of the vertices in $B$ such that \(|N_D(w_i)|=i - 1\) for any $i\in[x]$. Thus
\[
\overline{e}(B,D)=|B||D|-\sum_{i=1}^x(i-1)=\binom{x}{2}.
\]
This together with \eqref{Complement-G} implies that
\begin{align*}
e(\overline{G})&\ge(\beta-1-\textbf{1}_{uv})^2+4n-\frac12(x^2+x)-2-\textbf{1}_{uv}+\binom{x}{2}\\
&=n^2+2n-x-2\geq n^2+n-2,
\end{align*}
where the last inequality holds in view of $x\le\beta-\textbf{1}_{uv}=n$. This leads to a contradiction unless \(n =x= 2\). If \(n =x= 2\), then we let \(B = \{u_1,v_1\}\) and $D=\{w_1\}$. Clearly, $\{d(u_1),d(v_1)\}=\{2,3\}$ and $d(w_1)=1$. Without loss of generality, we may assume that \(d(u_1)=2\) and \(d(v_1)=3\). This implies that \(v_1vu_1u\) is a path of length three with equal-degree endpoints, a contradiction. This completes the proof of this lemma.
\end{proof}

The following lemma is given by Chen and Ma \cite{CM2025} (see Lemma 5), which plays a crucial role in our proof. We mention that the condition $n\ge5$ can be obtained from their proof.
\begin{lem}[Chen and Ma \cite{CM2025}]\label{Small-equal-degree}
For $n\geq5$, we have either \(\beta \geq \Delta - 1\) or \(\Delta \leq n + 1\), where $\Delta$ is the maximum degree of the vertices in  $G$.
\end{lem}
Now, we prove Theorem \ref{LZ} for \(n\geq5\).
\begin{proof}[\bf Proof of Theorem \ref{LZ} for \(n\geq5\)]Since \(e(G)\geq n^2 + n\), we have \(\Delta\geq n + 1\). This together with Lemmas \ref{Large-equal-degree} and \ref{Small-equal-degree} implies that $(\mathrm{i})$ $\beta=n+1$ and $G$ is isomorphic to \(K_{n,n + 1}\), or $(\mathrm{ii})$ $\beta\le n$ and \(\Delta=n + 1\). However, the second case implies that $e(G)\le(2n\beta+\Delta)/2<n^2+n$, a contradiction. Therefore, we conclude that \(G\) must be the complete bipartite graph \(K_{n,n + 1}\).
\end{proof}

%In order to apply Lemma 2.2, we consider the cases where \(n\leq4\) and \(n\geq5\) respectively. First, let's look at the case when \(n\geq5\). Since \(e(G)\geq n^2 + n\), then \(\Delta\geq n + 1\). Thus, Lemmas 2.1 and 2.2 imply that \(\Delta=n + 1\). If \(\beta\leq n\), then there is only one vertex of maximum degree in the graph \(G\), which contradicts \(e(G)\geq n^2 + n\). Therefore, \(\beta = n+1\) and the graph \(G\) can only be \(K_{n,n + 1}\).

 %By the inequality $\eqref{Complement-G}$, we know that
%\[e(\overline{G})\ge(\beta-1-\textbf{1}_{uv})^2+4n-\frac12(x^2+x)-2-\textbf{1}_{uv}\]
In what follows, we prove Theorem \ref{LZ} for \(n\in\{2,3,4\}\). By Lemma \ref{Large-equal-degree}, it suffices to prove Theorem \ref{LZ} for $\beta\le n$ by obtaining a contradiction. We first give an easy lemma.

\begin{lem}\label{2n+1 beta 3}
 For any integer $n\geq2$, we have $\beta\ge3$.
\end{lem}
\begin{proof}[\bf Proof]
Suppose that $\beta\leq2$. We claim that there exist two vertices $z_1,z_2$ with degrees $2n$ and $2n-1$, respectively. Otherwise, by the definition of $\beta$
\[
2e(G)\le 2+2+2+2+\left(\sum_{k=3}^{2n}k-(2n-1)\right)=2n^2-n+6.
\]
If \(n\geq3\), then this contradicts the fact that \(e(G)\geq n^{2}+n\). Therefore, we have \(n = \beta=2\), and the degree sequence is $(4,2,2,2,2)$. Clearly, there is a path of length three with equal-degree endpoints, a contradiction.
%At this time, the graph \(G\) is characterized and there exists a path of length three with equal-degree endpoints, which is a contradiction.
Since there exist two vertices $w_1,w_2$ with degrees exactly $\beta\le2$ by the definition of $\beta$, then at least one of $w_1z_1z_2w_2$ and $w_1z_2z_1w_2$ is a path of length three with equal-degree endpoints, a contradiction.
\end{proof}

We also need the following lemma obtained by Chen and Ma \cite{CM2025} (see Lemma 4).
\begin{lem}[Chen and Ma \cite{CM2025}]\label{d(u)=d(v)}
Let $G$ be a graph and $v \in V(G)$. If there exists a vertex $u \in N(v)$ with $|N(u)\cap N(v)|\ge2$, then $d(w)\neq d(u)$ for each $w \in N(v) \setminus \{u\}$.
\end{lem}

\begin{lem}\label{2n+1 beta 2n-2}
For \(n\in\{2,3,4\}\), we have $\Delta\leq 2n-2$.
\end{lem}

\begin{proof}[\bf Proof]
Suppose  that $\Delta\ge2n-1$. It follows from Lemma \ref{Large-equal-degree} that there exists a unique vertex $v_0$ in $G$ with $d(v_0)=\Delta$.
If \(\Delta=2n\), then there exist two vertices $v_1$ and $v_2$ in $N(v_0)$ with \(d(v_1)=d(v_2)=\beta\geq3 \) by Lemma \ref{2n+1 beta 3}. This implies that \(|N(v_0)\cap N(v_i)| \geq 2\) for each \(i\in\{1,2\}\) , which contradicts Lemma \ref{d(u)=d(v)}.
If \(\Delta = 2n - 1\), then there exists a unique vertex \(v_1\) such that \(v_1 \notin N(v_0)\cup\{v_0\}\). If \(d(v_1) = \beta\), then \(|N(v_0) \cap N(v_1)| = d(v_1) \geq 3\)  by Lemma \ref{2n+1 beta 3}. Suppose that \(v_2\) and \(v_3\) are two distinct vertices such that \(d(v_2) = \beta\) and \(v_3 \in N(v_0) \cap N(v_1)\). Then \(v_1v_3v_0v_2\) is a path of length three, and \(d(v_1) = d(v_2)\), a contradiction.
If \(d(v_1) \neq \beta\), then
\begin{align}\label{No-Three}
\text{there cannot exist three distinct vertices in \(G\) with the same degree at least \(3\)}. \tag{*}
\end{align}
 Otherwise, suppose that there exist three vertices \(v_2, v_3, v_4\) such that \(d(v_2) = d(v_3) = d(v_4) \geq 3\). Then \(N(v_i) \cap N(v_0) \neq \emptyset\)  for each \(i \in \{2, 3, 4\}\). Without loss of generality, suppose that there exists a vertex \(x \in N(v_2) \cap N(v_0)\). Then at least one of the paths \(v_2xv_0v_3\) and \(v_2xv_0v_4\) is a path of length three with equal-degree endpoints, a contradiction. Recall that $\beta\le n$ by Lemma \ref{Large-equal-degree}. This together with \eqref{No-Three} deduces that \((\mathrm{i})\) if \( n = 4 \), then \( 40 = 2(n^2 +n) \leq 2e(G) \leq 7 + 6  + 5 + 4 + 4 + 3 + 3 + 2 + 2 = 36 \), a contradiction; \((\mathrm{ii})\) if \( n = 3\), then \( 24 = 2(n^2 +n) \leq 2e(G) \leq 5  + 4 + 3 + 3 + 2 + 2 +2= 21 \), also a contradiction; and \((\mathrm{iii})\) if \( n = 2 \), then \(\beta \leq 2\)  by Lemma \ref{Large-equal-degree}, and this leads to a contradiction with Lemma \ref{2n+1 beta 3}. Thus, we complete the proof of this lemma.
\end{proof}

Now, we are in a position to give the proof of Theorem \ref{LZ} for \(n\in\{2,3,4\}\).
\begin{proof}[\bf Proof of Theorem \ref{LZ} for \(n\in\{2,3,4\}\)]
In view of Lemma \ref{2n+1 beta 2n-2}, we conclude that either \((\mathrm{i})\) $\Delta\le n+1$ for \(n\in\{2,3,4\}\), or \((\mathrm{ii})\) \(n = 4\) and \(\Delta = 6\). Recall that $\beta\le n$ by Lemma \ref{Large-equal-degree}. For the first case, we have
 \[
2(n^2 + n) \leq 2e(G) \leq 2n\cdot n+\Delta\le2n^2+n+1,
\]
a contradiction. For the second case, we have
\[
40 = 2(n^2 + n) \leq 2e(G) \leq 4(2n - 1) + 5 + 6 = 39,
\]
a contradiction. Therefore, we conclude that \(G\) must be the complete bipartite graph \(K_{n,n + 1}\).
\end{proof}
Finally, we complete the proof of Theorem \ref{LZ}.

\section{Proof of Theorem \ref{2n}}

Throughout this section, let $n\geq3$, and let $G$ be a graph on $2n$ vertices with at least $n^2 -1$ edges, such that $G$ contains no two vertices of the same degree joined by a path of length three. Our goal is to prove that $G$ must be the complete bipartite graph $K_{n+1,n-1}$. The proof is similar to that of Theorem \ref{LZ}.

Let $u$ and $v$ be two vertices in $G$ such that $d(u) = d(v) = \beta$, and let $\textbf{1}_{uv}=1$ if $uv\in E(G)$ and $\textbf{1}_{uv}=0$ if $uv\notin E(G)$. Set $B := N(u) \cap N(v)$,
$A_u := N(u)\setminus(\{v\} \cup B)$, $ A_v := N(v)\setminus(\{u\} \cup B)$ and $D :=V(G)\setminus(N(u)\cup N(v)\cup\{u,v\})$.
Let $|B|=x$. Then,
\(|A_u| =|A_v|=\beta-x-\textbf{1}_{uv}\) and \(|D| = (2n-2)-|B| -|A_u|-|A_v|=x-2(\beta-n-\textbf{1}_{uv})-2\).
Since $G$ contains no path of length three connecting $u$ and $v$, we have
\begin{equation}\label{ABuv'}
e(A_u, B) = e(A_v, B) = e(B) = e(A_u, A_v) = 0.
\end{equation}
It follows that
\begin{align*}
\overline{e}(V(G)\setminus D)&=\overline{e}(u,A_v)+ \overline{e}(v,A_u)+\overline{e}(A_u,A_v)+ \overline{e}(B,A_u)+ \overline{e}(B,A_v)+\overline{e}(B)+\overline{e}(u,v)\\
&=2(\beta-x-\textbf{1}_{uv})+(\beta-x-\textbf{1}_{uv})^2+2x(\beta-x-\textbf{1}_{uv})+\binom{x}{2}+(1-\textbf{1}_{uv})\\
&=(\beta+1-\textbf{1}_{uv})^2-\frac12(x^2+5x)-\textbf{1}_{uv}.
\end{align*}
This together with the fact $\overline{e}(\{u,v\},D)=2|D|$ implies that
\begin{align}\label{Complement-G'}
e(\overline{G})&\ge\overline{e}(V(G)\setminus D)+\overline{e}(\{u,v\},D)+\overline{e}(B,D)+\overline{e}(D)\notag\\
&=(\beta+1-\textbf{1}_{uv})^2-\frac12(x^2+5x)-\textbf{1}_{uv}+2(2n+x-2\beta+2\cdot\textbf{1}_{uv}-2)+\overline{e}(B,D)+\overline{e}(D)\notag\\
&=(\beta-1-\textbf{1}_{uv})^2+4n-\frac12(x^2+x)-4-\textbf{1}_{uv}+\overline{e}(B,D)+\overline{e}(D).
\end{align}
Let $c:=\beta-n-\textbf{1}_{uv}$. Note that
\[
|B|-|D|=2(\beta-n-\textbf{1}_{uv})+1=2c+2.
\]
If $c\ge1$, then \(|B|-|D|\geq4\). It follows that there exists \(\{u_1, u_2\} \subseteq B\) such that \(|N_{D}(u_1)| = |N_{D}(u_2)|\) by the pigeonhole principle, implying that \(d(u_1)=d(u_2)\). Let \(\gamma\) be the largest integer such that there are two vertices in  \(B\) that have \(\gamma\) neighbors in \(D\). Without loss of generality, we may assume that \(u_1\) and \(v_1\) are such two vertices. By \eqref{ABuv'}, \(u_1\) and \(v_1\) are not adjacent. Similarly, let $B_1 := N_D(u_1) \cap N_D(v_1)$,
$A_{u_1} := N_D(u_1)\setminus B_1$ and $A_{v_1} := N_D(v_1)\setminus B_1$. Write $|B_1|=y$. Then $|A_{u_1}|=|A_{v_1}|=\gamma-y$.
%$|A_{u_1}|=|A_{v_1}|=\gamma-y$ and $D_1 := D\setminus(A_{u_1}\cup A_{v_1}\cup B_1)$, $|D_1|=x-2c-1-2\gamma+y$.
By the assumption, there does not exist a path of length three connecting $u_1$ and $v_1$. Thus, we have
\[
e(A_{u_1}, B_1) = e(A_{v_1}, B_1) = e(B_1) = e(A_{u_1}, A_{v_1}) = 0.
\]
This together with the fact $y\le\gamma$ implies that
\begin{align*}
\overline{e}(D) &\geq \overline{e}(A_{u_1}, B_1) + \overline{e}(A_{v_1}, B_1) + \overline{e}(B_1) + \overline{e}(A_{u_1}, A_{v_1})\\
&= y(\gamma-y)+y(\gamma-y)+\binom{y}{2}+(\gamma-y)(\gamma-y)\\
&= \gamma^{2}-\frac{1}{2}y^{2}-\frac{1}{2}y\ge \binom{\gamma}{2}.
\end{align*}
By the definition of \(\gamma\), we deduce that
\begin{align*}
e(B,D)= \sum_{v \in B} |N_D(v)| &\leq  \gamma\left(|B|-(|D|-\gamma)\right) + \sum_{k = \gamma+1}^{|D|} k  \\
&=\gamma|B|+\frac12(|D|-\gamma)(|D|-\gamma+1).
\end{align*}
Note that $|B|=x$ and $|D|=x-2c-2$. Thus
\begin{align*}
\overline{e}(B,D) &=|B||D|-e(B, D)\\
%&\ge |B|(|D|-\gamma)-\frac12(|D|-\gamma)(|D|-\gamma+1)\\
&\ge\frac12(|D|-\gamma)(2|B|-|D|+\gamma-1)\\
&=\binom{x+1}{2}-\binom{\gamma}{2}-2c^2-3c-(2c+2)\gamma-x-1\\
&\ge\binom{x+1}{2}-\binom{\gamma}{2}+2c^2+5c+3-(2c+3)x,
%&=\binom{x}{2}-(\beta-n-\textbf{1}_{uv})^2-(2\gamma+1)(\beta-n-\textbf{1}_{uv})-\binom{\gamma+1}{2}\\
%&=\binom{x+1}{2}-x-(\beta-n-\textbf{1}_{uv})^2-(2\gamma+1)(\beta-n-\textbf{1}_{uv})-\binom{\gamma}{2}-\gamma\\
%&=\binom{x+1}{2}-\binom{\gamma}{2}-(\beta-n-\textbf{1}_{uv})^2-(\beta-n-\textbf{1}_{uv})-\gamma(2(\beta-n-\textbf{1}_{uv})+1)-x
%\\&\geq (x-2c-1)x-(2c + 1 + \gamma)\gamma-\frac{(x - 2c + \gamma)(x - 2c - 1 - \gamma)}{2}\\
%&=\frac{1}{2}x^{2}-\frac{1}{2}x - 2c^{2}-2\gamma c - c-\frac{1}{2}\gamma^{2}-\frac{1}{2}\gamma.
\end{align*}
where the last inequality follows from the fact that $\gamma\le|D|=x-2c-2$.
In view of \eqref{Complement-G'}, we have
\begin{align}\label{c-Complement-G'}
e(\overline{G})&\ge\overline{e}(V(G)\setminus D)+\overline{e}(\{u,v\},D)+\overline{e}(B,D)+\overline{e}(D)\notag\\
&\ge(\beta-1-\textbf{1}_{uv})^2+4n-4-\textbf{1}_{uv}+2c^2+5c+3-(2c+3)x\notag\\
&\ge (n^{2}-n+1)+c^2-1-\textbf{1}_{uv},
\end{align}
where the last inequality holds in view of $x\le\beta-\textbf{1}_{uv}=n+c$.

\begin{lem}\label{beta=n}
    For any integer $n\ge3$, we have $\beta \leq n + 1$. In particular, $G$ is isomorphic to \(K_{n+1,n - 1}\) if $\beta=n + 1$.
%Let $n>1$. We have $\beta \leq n + 1$. In addition, when \(\beta=n + 1\), the graph \(G\) is \(K_{n+1,n -1}\).
\end{lem}
\begin{proof}
Suppose that $\beta\ge n+1$. We aim to prove that  $G$ is isomorphic to \(K_{n+1,n - 1}\). Let $u$ and $v$ are two vertices in $G$ such that $d(u) = d(v) = \beta$. Since $\beta\ge n+1$,  we have $c=\beta-n-\textbf{1}_{uv}\ge1-\textbf{1}_{uv}$. If $\textbf{1}_{uv}=0$, then $e(\overline{G})>n^2-n+1$ by \eqref{c-Complement-G'} unless $c=1$, a contradiction. If $c=1$, then we conclude that \(x=n + c=n+1\) and \(\gamma=x-2c-2 =n-2\), implying $e(B,D)=|B||D|$. Thus, \(G\) is isomorphic to \(K_{n+1,n -1}\). If $\textbf{1}_{uv}=1$, then $|B|-|D|\ge2$. Thus, there exist \(u_1, v_1\in B\) such that \(d(u_1)=d(v_1)\) by the pigeonhole principle, and \(u_1uvv_1\) is a path of length \(3\). This leads to a contradiction.
\end{proof}

\begin{lem}\label{beta=3}
   For any integer $n\geq3$, we have $3\leq \beta\leq \Delta$.
\end{lem}
\begin{proof}
    If $\beta\leq2$, then there exist two vertices $z_1,z_2$ with degrees $2n-1$ and $2n-2$, respectively. Otherwise, by the definition of $\beta$
\[
2e(G)\le (2n-(2n-4))2+\left(\sum_{k=3}^{2n-1}k-(2n-2)\right)=2n^2-3n+7.
\]
If \(n\geq4\), then this contradicts the fact that \(e(G)\geq n^{2}-1\). Therefore, \(n = 3\), and the equality holds. At this time, the graph \(G\) can be characterized and there exists a path of length three with equal-degree endpoints, a contradiction. Since there exist two vertices $w_1,w_2$ with degrees exactly $\beta$ by the definition of $\beta$, then at least one of $w_1z_1z_2w_2$ and $w_1z_2z_1w_2$ is a path of length three with equal-degree endpoints, a contradiction.
\end{proof}
We also obtain the following version of Lemma \ref{Small-equal-degree} for $G$ with even number of vertices. Since its proof is very similar as that of  Lemma \ref{Small-equal-degree}, we present its details in Appendix.

\begin{lem}\label{beta n+2}
For any integer $n\geq 6$, we have either $\beta \geq \Delta-1$ or $\Delta\leq n+2$.
\end{lem}

For small $n$, we also have the following upper bound of $\Delta$.
\begin{lem}\label{345}
    For $n\in \{3,4,5\}$, we have $\Delta\leq2n-3$.
\end{lem}
\begin{proof}
    Suppose  that $\Delta\ge2n-2$. It follows from Lemma \ref{beta=n} that there exists a unique vertex $v_0$ in $G$ with $d(v_0)=\Delta$.
If \(\Delta=2n-1\), then there exist two vertices $v_1$ and $v_2$ in $N(v_0)$ with \(d(v_1)=d(v_2)=\beta\geq3 \) by Lemma \ref{2n+1 beta 3}. This implies that \(|N(v_0)\cap N(v_i)| \geq 2\) for each \(i\in\{1,2\}\) , which contradicts Lemma \ref{d(u)=d(v)}\footnote{Note that Lemma \ref{d(u)=d(v)} also holds for graphs with an even number of vertices.}.
If \(\Delta = 2n - 2\), then there exists a unique vertex \(v_1\) such that \(v_1 \notin N(v_0)\cup\{v_0\}\). If \(d(v_1) = \beta\), then \(|N(v_0) \cap N(v_1)| = d(v_1) \geq 3\)  by Lemma \ref{beta=3}. Suppose that \(v_2\) and \(v_3\) are two distinct vertices such that \(d(v_2) = \beta\) and \(v_3 \in N(v_0) \cap N(v_1)\). Then \(v_1v_3v_0v_2\) is a path of length three, and \(d(v_1) = d(v_2)\), a contradiction.
%If \(\Delta=2n-1\), then there exist a unique vertex $v_0$ in $G$ with $d(v_0)=2n-1$. By Lemma \ref{beta=3} we assume that \(d(v_1)=d(v_2)=\beta\geq3 \). Then for \(i\in\{1,2\}\), \(v_i\in N(v_0)\), and \(|N(v_0)\cap N(v_i)| \geq 2\), which contradicts Lemma \ref{d(u)=d(v)}\footnote{Note that Lemma \ref{d(u)=d(v)} also holds for graphs with an even number of vertices.}. If \(\Delta = 2n - 2\), then there exists a unique vertex \(v_0\) in \(G\) with \(d(v_0) = 2n - 2\), and a unique vertex \(v_1\) such that \(v_1 \notin N(v_0)\). If \(d(v_1) = \beta\), then by Lemma \ref{beta=3}, \(|N(v_0) \cap N(v_1)| = d(v_1) \geq 3\). Without loss of generality, let \(v_2\) and \(v_3\) be two distinct vertices such that \(d(v_2) = \beta\) and \(v_3 \in N(v_0) \cap N(v_1)\). Then \(v_1v_3v_0v_2\) is a path of length 3, and \(d(v_1) = d(v_2)\), leading to a contradiction.
If \(d(v_1) \neq \beta\), then there cannot exist three distinct vertices in \(G\) with the same degree at least 3. Otherwise, suppose that there exist three vertices \(v_2, v_3, v_4\) such that \(d(v_2) = d(v_3) = d(v_4) \geq 3\). Then \(N(v_i) \cap N(v_0) \neq \emptyset\)  for each \(i \in \{2, 3, 4\}\). Without loss of generality, there exists a vertex \(x \in N(v_2) \cap N(v_0)\). Then at least one of the paths \(v_2xv_0v_3\) and \(v_2xv_0v_4\) is a path of length three with equal-degree endpoints, a contradiction.
If \( n = 5 \), then \( 48 = 2(n^2 - 1) \leq 2e(G) \leq 8 + 7 + 6 + 5 + 5 + 4 + 4 + 3 + 3 + 2 = 47 \), a contradiction. If \( n = 4 \), then \( 30 = 2(n^2 - 1) \leq 2e(G) \leq 6 + 5 + 4 + 4 + 3 + 3 + 2 + 2 = 29 \), also a contradiction. If \( n = 3 \), then \( 16 = 2(n^2 - 1) \leq 2e(G) \leq 4 + 3 + 3 + 2 + 2 + 2 = 16 \), which implies \( d(v_1) = 2 \). Without loss of generality, let \( v_2 \) and \( v_3 \) be two vertices such that \( d(v_2) = d(v_3) = 2 \). Let \( y \in N(v_1) \cap N(v_0) \). Then at least one of the paths \( v_1yv_0v_2 \) and \( v_1yv_0v_3 \) is a path of length three with equal-degree endpoints, a contradiction.
\end{proof}

Combining Lemmas  \ref{beta=n}, \ref{beta n+2} and \ref{345}, we deduce that \(\Delta \leq n + 2\),  \(\beta \leq n\), and
\begin{itemize}
    \item \(n \geq 5\) if \(\Delta = n + 2\);
    \item \(n \geq 4\) if \(\Delta = n + 1\);
    \item \(n \geq 3\) if \(\Delta = n\).
\end{itemize}
Note that $\Delta\ge n$. In what follows, we prove Theorem \ref{2n} by discussing the possible values of \(\Delta\).

\textbf{Case 1:} \(\Delta = n + 2\). There exists a unique vertex \(v_0\) in \(G\) with \(d(v_0) = n + 2\) and \(|\overline{N}(v_0)| = n - 3\). Moreover, the graph \(G\) has at least \(2n - 7\) vertices with degree \(n\). Otherwise, we have
\[
2(n^2 - 1) \leq 2e(G) \leq(n + 2) + (n + 1) + (2n - 8)n + 6(n - 1) = 2n^2 - 3,
\]
which is a contradiction. Since \(n \geq 5\) and \(2n - 7 - (n - 3) = n - 4\), there must exist a vertex \(v_1\) such that \(d(v_1) = n\) and \(v_1 \in N(v_0)\). If there exists another vertex \(v_2\) such that \(d(v_2) = n\) and \(v_2 \in N(v_0)\), then for \(i \in \{1, 2\}\),
\[
|N(v_0) \cap N(v_i)| =|N(v_i) \setminus (\{v_0\} \cup \overline{N}(v_0))|\geq n - 1 - (n - 3) = 2,
\]
which contradicts Lemma \ref{Small-equal-degree}. This implies that there is exactly one vertex in \(N(v_0)\) with degree \(n\). Therefore, there must exist a vertex \(v_2\) such that \(d(v_2) = n\) and \(v_2 \in \overline{N}(v_0)\). Since \(n \geq 5\), we have
\[
|N(v_2) \cap N(v_0)|  =|N(v_2) \setminus (\overline{N}(v_0) \setminus \{v_2\})|\geq n - (n - 3-1) = 4.
\]
Thus, there must be a vertex \(x \in N(v_0)\), different from \(v_1\), such that \(x \in N(v_2) \cap N(v_0)\). Then the path \(v_2xv_0v_1\) is a path of length 3 with \(d(v_2) = d(v_1)\), leading to a contradiction.

\textbf{Case 2:} \(\Delta = n + 1\). There exists a unique vertex \(v_0\) in \(G\) with \(d(v_0) = n + 1\) and \(|\overline{N}(v_0)| = n - 2\). Moreover, the graph \(G\) has at least \(2n - 4\) vertices with degree \(n\). Otherwise, we have
\[
2(n^2 - 1) \leq 2e(G) \leq (n + 1) + (2n - 5)n + 4(n - 1) = 2n^2 - 3,
\]
which is a contradiction. Since \(n \geq 4\) and \(2n - 4 - (n - 2) = n - 2 \geq 2\), there must exist vertices \(v_1\) and \(v_2\) such that \(d(v_1) = d(v_2) = n\) and \(v_1, v_2 \in N(v_0)\). If there exists a vertex \(v_3 \neq v_1, v_2\) such that \(d(v_3) = n\) and \(v_3 \in N(v_0)\), then for each \(i \in \{1, 2, 3\}\),
\[
|N(v_i) \cap N(v_0)| = |N(v_i) \setminus (\{v_0\} \cup \overline{N}(v_0))| \geq (n - 1) - (n - 2) = 1.
\]
Without loss of generality, there exists a vertex \(x \in N(v_1) \cap N(v_0)\). Then at least one of the paths \(v_1xv_0v_2\) and \(v_1xv_0v_3\) is a path of length three with equal-degree endpoints, a contradiction. This implies that there are exactly two vertices in \(N(v_0)\) with degree \(n\). Therefore, there must exist a vertex \(v_3\) such that \(d(v_3) = n\) and \(v_3 \in \overline{N}(v_0)\). Since \(n \geq 4\), we have
\[
|N(v_3) \cap N(v_0)| = |N(v_3) \setminus (\overline{N}(v_0) \setminus \{v_3\})| \geq n - (n - 2 - 1) = 3.
\]
Thus, there must be a vertex \(x \in N(v_0)\), different from \(v_1\) and \(v_2\), such that \(x \in N(v_3) \cap N(v_0)\). Then the path \(v_3xv_0v_1\) is a path of length 3 with \(d(v_3) = d(v_1)\), a contradiction.

\textbf{Case 3:} \(\Delta = n\). The graph \(G\) has at least \(2n - 2\) vertices with degree \(n\). Otherwise, we have
\[
2(n^2 - 1) \leq 2e(G) \leq (2n - 3)n + 3(n - 1) = 2n^2 - 3,
\]
a contradiction. Since \(n \geq 3\) and \(2n - 2 - (n + 1) = n - 3\), there must exist two non-adjacent vertices \(v_1\) and \(v_2\) such that \(d(v_1) = d(v_2) = n\) unless \(n = 3\) and the four vertices of degree 3 in \(G\) form a complete graph \(K_4\). In the latter case, there clearly exists a path of length 3 with equal degree endpoints, a contradiction. Therefore, \(v_1\) and \(v_2\) are non-adjacent. Note that
\[
|N(v_1) \cap N(v_2)| = |N(v_1) \setminus (\overline{N}(v_2) \cup \{v_1\})| \geq n - (n - 1 - 1) = 2.
\]
Without loss of generality, let \(v_3, v_4 \in N(v_1) \cap N(v_2)\). If either \(v_3\) or \(v_4\) has degree \(n\), then at least one of the paths \(v_1v_4v_2v_3\) or \(v_1v_3v_2v_4\) is a path of length 3 with equal-degree endpoints, a contradiction. If both \(v_3\) and \(v_4\) do not have degree \(n\), then we must have \(d(x) = n\) for any vertex \(x \in N(v_1) \setminus \{v_3, v_4\}\) as \(G\) has at least \(2n - 2\) vertices of degree \(n\). Then the path \(v_2v_3v_1x\) is a path of length 3 with \(d(x) = d(v_2) = n\), a contradiction.
%\textbf{Case 4:} \(\Delta = n - 1\). Since \(n \geq 3\), we have:
%\[2(n^2 - 1) \leq 2e(G) \leq (n - 1) \cdot 2n = 2n^2 - 2n,\]
%leading to a contradiction.
Finally, we complete the proof of Theorem \ref{2n}.

\section{Concluding remarks}
In this paper, we prove that, for all $n\ge2$, every \((2n + 1)\)-vertex graph with at least \(n^{2}+n+1\) edges contains two vertices of the same degree joined by a path of length three. This provides a complement of the result established by Chen and Ma \cite{CM2025}, and thus resolve the question of Erd\H{o}s and Hajnal completely.

In the end of this paper, we address the following general problems raised by Chen and Ma \cite{CM2025} on paths with equal-degree endpoints. For any positive integers \(\ell\) and \(n\), let \(p_{\ell}(n)\) denote the maximum number of edges in an \(n\)-vertex graph that contains no two vertices of equal degree connected by a path of length \(\ell\). Chen and Ma \cite{CM2025} determined \(p_{\ell}(n)\) for each $\ell\in\{1,2,3\}$. For odd $\ell\ge5$, they conjectured that the complete bipartite graph \(K_{n,n + 1}\) remains extremal.
\begin{conj}[Chen and Ma \cite{CM2025}]
For any odd integer \(\ell\geq5\) and sufficiently large \(n\), it holds that
\[p_{\ell}(2n + 1)=n^{2}+n.\]
\end{conj}

For even $\ell\ge2$, the authors show that $p_{\ell}(2n)\ge (n^2+n)/2$ by considering the half graph. However, it is unclear whether this is optimal. They tend to believe that the behavior differs significantly between odd and even $\ell$, and therefore proposed the following problem.
\begin{prob}[Chen and Ma \cite{CM2025}]
Determine the exact value of \(p_{\ell}(2n)\) for all even \(\ell\) and sufficiently large \(n\).
\end{prob}

\section*{Appendix: Proof of Lemma \ref{beta n+2}}\label{appp}

%\begin{lem}[\bf Lemma 3.3]
%Let $n\geq 6$. We have either $\beta \geq \Delta-1$ or $\Delta\leq n+2$.
%\end{lem}
\begin{proof}[\bf  Proof of Lemma \ref{beta n+2}]
    Suppose for a contradiction that $\beta \leq \Delta-2$ and $\Delta\geq n+3$.
Then there exists a unique vertex $v_0$ in $G$ with $d(v_0)=\Delta$.
We divide $N(v_0)$ into the following two parts
\begin{align*}
A:=\{ v\in N(v_0) ~|~ d(v)\geq 2n-\Delta+2 \}
\,\:\text{and}\,\:
B:=\{ v\in N(v_0)~ |~ d(v)\leq 2n-\Delta +1 \}.
\end{align*}
Note that any $v\in A$ satisfies that $d(v)\in [2n-\Delta+2,\Delta-1]$ as $\beta \leq \Delta-2$ and $\Delta\geq n+3$.
Note that $|N(v)\cap N(v_0)|=|N(v)\setminus (\{v_0\}\cup \overline{N}(v_0))|\geq (2n-\Delta+2)-(2n-\Delta)=2$ for any vertex $v\in A$. This yields that $v$ has at least two neighbors in $N(v_0)$.
By Lemma \ref{d(u)=d(v)}\footnote{Recall that Lemma \ref{d(u)=d(v)} still holds for graphs with an even number of vertices.}, we know that $d(v)$ is distinct from the degree of any other vertex in $N(v_0)$.
In particular, all vertices in $A$ have distinct degrees.
Since there are at most $2\Delta - 2n - 2$ possible values for the degree of a vertex in $A$, we have $2\Delta - 2n - 2 \geq |A| = \Delta - |B|$. This means that
\begin{align}\label{equ:|B|>=}
|B|\geq 2n-\Delta +2.
\end{align}

In the remainder of the proof, we estimate the number of edges in $G$ to derive a final contradiction.
As $\sum_{v\in B} d(v)$ is evidently bounded from above by definition,
our crucial step is to maximize $\sum_{v \in A \cup \overline{N}(v_0)} d(v)$, subject to
\begin{itemize}
\item[(a)] all vertices in $A$ have distinct degrees;
\item[(b)] there are no two vertices in $A \cup \overline{N}(v_0)$ with the same degree greater than $\beta$.
\end{itemize}
Let $\mathcal{A}$ be the sequence of degrees $d(v)$ for all $v \in A$, and let $\mathcal{B}$ be the sequence of degrees $d(v)$ for all $v \in \overline{N}(v_0)$. We transform this to the following maximization problem:
Given integers $n,\Delta,\beta$ and $|B|$,
$$\mbox{ determine } \lambda(n,\Delta,\beta, |B|):=\max_{\mathcal{A}, \mathcal{B}} \sum_{k \in \mathcal{A} \cup \mathcal{B}} k,$$
where
\begin{itemize}
\item[(1)] $\mathcal{A}$ is a sequence of $\Delta-|B|$ distinct integers in $\{1, 2, \dots, \Delta - 1\}$,
\item[(2)] $\mathcal{B}$ is a sequence of $2n-1 - \Delta$ integers in $\{0,1,\dots, \Delta - 1\}$, and
\item[(3)] no two elements in $\mathcal{A} \cup \mathcal{B}$ share the same value greater than $\beta$.
\end{itemize}
Note that when $\lambda(n, \Delta, \beta, |B|)$ is maximized, every element in $\mathcal{B}$ is at least $\beta$, and we have
\begin{align}\label{equ:opt}
\sum_{v \in A \cup \overline{N}(v_0)} d(v)\leq \lambda(n,\Delta,\beta, |B|).
\end{align}
Now, we proceed our proof according to the following three cases.

\bigskip
{\bf Case 1}. $|B|\leq \beta$.
\medskip

We have $|A| = \Delta - |B| \geq \Delta - \beta$. In this case, we claim that
\begin{align}\label{equ:case1}
\lambda(n, \Delta, \beta, |B|) = \sum_{k = |B|}^{\Delta - 1} k + (2n - 1 - \Delta)\beta.
\end{align}
We first assert that  $\mathcal{A} \supseteq \{\beta, \beta + 1, \ldots, \Delta - 1\}$ when the maximum is achieved.
Otherwise, there exist $k \in \mathcal{A}$ and $\ell \notin \mathcal{A}$ with $k < \beta \leq \ell \leq \Delta - 1$.
If $\ell$ is not included in $\mathcal{B}$, we replace $k$ with $\ell$ in $\mathcal{A}$.
If $\ell$ is included in $\mathcal{B}$, we replace $k$ with $\ell$ in $\mathcal{A}$ and also replace $\ell$ with $\beta$ in $\mathcal{B}$.
In both cases, conditions (1), (2), and (3) remain satisfied,
however the sum $\sum_{k \in \mathcal{A} \cup \mathcal{B}} k$ strictly increases, a contradiction.
Thus, we can easily deduce that all elements in $\mathcal{B}$ must be $\beta$, and then $\mathcal{A} = \{|B|, |B| + 1, \ldots, \Delta - 1\}$. This proves \eqref{equ:case1}.

Now using \eqref{equ:opt} and \eqref{equ:case1}, we can derive that
\begin{align}\notag
\sum_{v\in V(G)}d(v)&= \Delta+ \sum_{v\in A\cup \overline{N}(v_0)}d(v) + \sum_{v\in B}d(v)\le \Delta+\left(\sum_{k=|B|}^{\Delta-1}k + (2n-1-\Delta)\beta\right) + |B|(2n-\Delta+1)
\\ &= -\frac{1}{2}|B|^2+ (2n-\Delta+\frac{3}{2}) |B|+(2n-1-\Delta)\beta+\frac{1}{2}\Delta^2+\frac{1}{2}\Delta.\label{equ:Case1-1}
\end{align}
Here, we treat the above rightmost expression in \eqref{equ:Case1-1} as a quadratic polynomial $f_1(|B|)$ in the variable $|B|$.
By \eqref{equ:|B|>=}, we have $|B|\ge 2n-\Delta +2$, and $f_1(|B|)$ is decreasing in this range.
Thus, we deduce that
\begin{align}\notag
\sum_{v\in V(G)}d(v)\le f_1(|B|) &\le f_1(2n-\Delta +2)=2n^2+3n-2n\Delta+\Delta^2-\Delta+1+(2n-1-\Delta)\beta.
\end{align}
Since $\sum_{v\in V(G)}d(v)=2e(G)\geq 2(n^2-1)$, we further derive that
\begin{align}\label{18}
(2n-1-\Delta)\beta-2n\Delta+3n+\Delta^2 -\Delta+3\ge 0.
\end{align}
Clearly $\Delta\leq 2n-1$.
If $2n-1-\Delta=0$, then inequality \eqref{18} is equivalent to $n\le 5$, a contradiction to our assumption that $n\geq 6$.
It follows that $2n-1-\Delta\geq 1$. Then \eqref{18} implies that
\begin{align*}
\beta \ge \frac{2n\Delta-3n-\Delta^2 +\Delta-3}{2n-1-\Delta}= \Delta-2+\frac{n-5}{2n-1-\Delta}> \Delta-2,
\end{align*}
which contradicts our assumption that $\beta \le \Delta-2$.
This finishes the proof of Case 1.

\bigskip
{\bf Case 2}. $|B|\geq \beta+1$ and $\Delta+|B|-2n\ge \beta$.
\medskip

Note that the combined sequence $\mathcal{A} \cup \mathcal{B}$ has $2n-1 - |B|$ elements.
Since $\Delta + |B| - 2n \geq \beta $,
condition (3) implies that the optimal configuration for $\mathcal{A} \cup \mathcal{B}$ is $\{\Delta - 1, \Delta - 2, \ldots, \Delta + |B| - 2n+1\}$,
for which conditions (1) and (2) remain satisfied. Thus
\begin{align}\label{equ:case2}
\lambda(n, \Delta, \beta, |B|) = \sum_{k = \Delta + |B| - 2n+1}^{\Delta - 1} k.
\end{align}
Similarly to the previous case, using \eqref{equ:opt} and \eqref{equ:case2}, we can derive that
\begin{align}\notag
\sum_{v\in V(G)}d(v)&= \Delta+ \sum_{v\in A\cup \overline{N}(v_0)}d(v) + \sum_{v\in B}d(v)\le \Delta+ \left(\sum_{k=\Delta+ |B|-2n+1}^{\Delta-1}k\right)+ |B|(2n-\Delta+1)
\\ &=-\frac{1}{2}|B|^2+ (4n-2\Delta+\frac{1}{2} ) |B| -2n^2+2n\Delta+n.\label{equ:Case2-1}
\end{align}
Now we treat the above rightmost expression in \eqref{equ:Case2-1} as a quadratic polynomial $f_2(|B|)$ in the variable $|B|$.
Since $|B|$ is an integer, it is easy to deduce that
\begin{align}\label{equ:Case2-2}
\sum_{v\in V(G)}d(v)\le f_2(|B|)&\le f_2(4n- 2\Delta)=2\Delta^2-(6n+1)\Delta+6n^2+3n.
\end{align}
Since $\sum_{v\in V(G)}d(v)=2e(G)\ge 2(n^2-1)$, we have
\begin{align}\label{equ:Case2-3}
2\Delta^2-(6n+1)\Delta+4n^2+3n+2\ge 0.
\end{align}
Consider the expression above as a quadratic polynomial in $\Delta$, and elementary calculations show that either
\begin{align}\label{equ:case2-Delta}
\Delta\ge \frac{6n+1+\sqrt{4n^2-12n-15}}{4}>2n-1, \mbox{ ~or~ } \Delta\le \frac{6n+1-\sqrt{4n^2-12n-15}}{4}< n+3.
\end{align}
The first inequality cannot hold, and the second one contradicts our assumption that $\Delta \geq n + 3$. This completes the proof for Case~2.

\bigskip
{\bf Case 3}. $|B|\geq \beta+1$ and $\Delta+|B|-2n\leq \beta-1$.
\medskip

With the optimization reasoning from the previous cases in mind,
it is straightforward to see that the optimal configuration in this case consists of all elements of $\mathcal{A}$ and some elements of $\mathcal{B}$ forming $\{\Delta - 1, \Delta - 2, \ldots, \beta + 1\}$,
while the remaining elements of $\mathcal{B}$ are all $\beta$.
Therefore
\begin{align}\label{equ:case3}
\lambda(n, \Delta, \beta, |B|)=\sum_{k=\beta+1}^{\Delta-1}k + (2n -|B| -\Delta+\beta)\beta.
\end{align}
Then using \eqref{equ:opt} and \eqref{equ:case3}, we have
\begin{align}\notag
\sum_{v\in V(G)}d(v)&= \Delta+ \sum_{v\in A\cup \overline{N}(v_0)}d(v) + \sum_{v\in B}d(v)\\ \notag
&\leq \Delta+\left(\sum_{k=\beta+1}^{\Delta-1}k + (2n -|B| -\Delta+\beta)\beta\right) + |B|(2n-\Delta+1)\\ \label{equ:Case3-1}
&= \Delta+\sum_{k=\beta+1}^{\Delta-1}k + (2n-\Delta+\beta)\beta + |B|(2n-\Delta+1-\beta).
\end{align}
Now according to this formula, we further divide Case~3 into the following two subcases.

\bigskip

{\bf Subcase 1}. $2n-\Delta+1-\beta\geq 0$.

Since $|B|\leq \Delta$, we deduce from \eqref{equ:Case3-1} that
\begin{align}\notag
\sum_{v\in V(G)}d(v) &\le \Delta+\sum_{k=\beta+1}^{\Delta-1}k + (2n -\Delta+\beta)\beta + \Delta(2n-\Delta+1-\beta)
\\ &=\frac{1}{2}\beta^2+( 2n-2\Delta-\frac{1}{2}) \beta+2n\Delta-\frac{1}{2}\Delta^2+\frac{3}{2}\Delta.\label{equ:Case3-sub1}
\end{align}
Consider the above rightmost expression in \eqref{equ:Case3-sub1} as a quadratic polynomial $f_3(\beta)$ in the variable $\beta$.
Using Lemma \ref{beta=3} and the assumption of this subcase, $3\leq \beta \leq 2n-\Delta+1$. Therefore,
\begin{align}\label{equ:max-1}
\sum_{v\in V(G)}d(v)\le f_3(\beta) \le \max \{ f_3(3), f_3(2n-\Delta+1) \}.
\end{align}

First, suppose the above maximum is achieved by $f_3(3)$. Then
\begin{align*}
\sum_{v\in V(G)}d(v)\le f_3(3)=-\frac{1}{2}\Delta^2 + \left(2n-\frac{9}{2}\right) \Delta+6n+3.
\end{align*}
Since $\sum_{v\in V(G)}d(v)=2e(G)\geq 2(n^2-1)$, we derive that
\begin{align*}
-\frac{1}{2}\Delta^2 + \left(2n-\frac{9}{2}\right) \Delta-2n^2+6n+5\ge 0.
\end{align*}
View the expression above as a quadratic polynomial in the variable \(\Delta\).
Then its discriminant
$\left(2n - \frac{9}{2}\right)^2 - 2(2n^2 - 6n - 5) = -6n + \frac{121}{4}$
must be non-negative, that is,
$n\leq \frac{121}{24}$, a contradiction.
Now we may assume that the maximum in \eqref{equ:max-1} is achieved by $f_3(2n-\Delta+1)$, implying
\begin{align*}
\sum_{v\in V(G)}d(v)\le f_3(2n-\Delta+1)=2\Delta^2-(6n+1)\Delta+6n^2+3n .
\end{align*}
Note that this inequality is identical to \eqref{equ:Case2-2}.
We complete the proof of this subcase through the same analysis via \eqref{equ:Case2-3} and \eqref{equ:case2-Delta}.

\bigskip
{\bf Subcase 2}. $2n-\Delta+1-\beta< 0$.

In this subcase, the last term of the rightmost expression in \eqref{equ:Case3-1} increases as $|B|$ decreases.
Using the condition $|B|\ge \beta +1$, we deduce from \eqref{equ:Case3-1} that
\begin{align}\notag
\sum_{v\in V(G)}d(v) &\le \Delta+\sum_{k=\beta+1}^{\Delta-1}k + (2n -\Delta-1)\beta + (\beta +1)(2n-\Delta+1)
\\ &=-\frac{1}{2}\beta^2+ (4n-2\Delta-\frac{1}{2}) \beta +2n+\frac{1}{2}\Delta^2- \frac{1}{2}\Delta+1.\label{equ:Case3-sub2}
\end{align}
View the expression above as a quadratic  polynomial $f_4(\beta)$ in the variable $\beta$.
Since $\beta$ is an integer, \eqref{equ:Case3-sub2} implies that
\begin{align}\notag
\sum_{v\in V(G)}d(v) \le f_4(\beta) \le f_4(4n-2\Delta-1)=\frac{5}{2} \Delta^2 - (8n-\frac{1}{2})\Delta +8n^2+1.
\end{align}
Since $\sum_{v\in V(G)}d(v)\geq 2(n^2-1)$, through elementary computation, this leads to that either
\begin{align}\notag
\Delta\ge \frac{8n-\frac{1}{2}+\sqrt{4n^2-8n-\frac{119}{4}}}{5} >2n-1 \mbox{~ or ~}
\Delta\le \frac{8n+\frac{1}{2}-\sqrt{4n^2-8n-\frac{119}{4}}}{5}< \frac{4}{3}n+\frac{1}{2}.
\end{align}
The former inequality clearly cannot hold, so the later inequality $\Delta<\frac{4}{3}n+\frac{1}{2}$ holds.
This, together with the assumption that \(\beta \leq \Delta - 2\), implies that
\begin{align*}
\beta\le \Delta-2< 4n-2\Delta -\frac{1}{2}.
\end{align*}
Now using \eqref{equ:Case3-sub2} again, we deduce that
\begin{align*}
\sum_{v\in V(G)}d(v) \leq f_4(\beta) \leq f_4(\Delta-2) =-2\Delta^2+(4n+5)\Delta-6n.
\end{align*}
Substituting $\sum_{v\in V(G)}d(v)\ge 2(n^2-1)$, we derive
\begin{align}\label{equ:Case3-sub2-disc}
-2\Delta^2+(4n+5)\Delta-2n^2-6n+2\ge 0.
\end{align}
The discriminant of the leftmost expression above equals $(4n+5)^2-8(2n^2+6n-2)=41-8n$, which is strictly negative.
This is a contradiction to \eqref{equ:Case3-sub2-disc} for its non-negativity, proving Subcase~2.
Finally, we complete the proof of Lemma \ref{beta n+2}.
\end{proof}

\end{document}